\documentclass[11pt,oneside,english]{amsart}
\usepackage[T1]{fontenc}
\usepackage{float}
\usepackage{units}
\usepackage{mathrsfs}
\usepackage{amsthm}
\usepackage{amstext}
\usepackage{amssymb}
\usepackage{graphicx}
\usepackage{esint}

\makeatletter

\numberwithin{equation}{section}
\numberwithin{figure}{section}
\theoremstyle{plain}
\newtheorem{thm}{\protect\theoremname}
  \theoremstyle{plain}
  \newtheorem{conjecture}[thm]{\protect\conjecturename}
  \theoremstyle{plain}
  \newtheorem{prop}[thm]{\protect\propositionname}
  \theoremstyle{plain}
  \newtheorem{lem}[thm]{\protect\lemmaname}
  \theoremstyle{definition}
  \newtheorem{defn}[thm]{\protect\definitionname}
  \theoremstyle{definition}
  \newtheorem{example}[thm]{\protect\examplename}

\usepackage{ifpdf}
\ifpdf

\IfFileExists{lmodern.sty}
 {\usepackage{lmodern}}{}

\fi 

\author{Yotam Smilansky} \address{Raymond and Beverly Sackler School of Mathematical Sciences, Tel Aviv University, Tel Aviv 69978, Israel} \email{yotamsmi@post.tau.ac.il}

\date{\today}

\makeatother

\usepackage{babel}
  \providecommand{\conjecturename}{Conjecture}
  \providecommand{\definitionname}{Definition}
  \providecommand{\examplename}{Example}
  \providecommand{\lemmaname}{Lemma}
  \providecommand{\propositionname}{Proposition}
\providecommand{\theoremname}{Theorem}

\begin{document}
\global\long\def\bbR{\mathbb{R}}

\global\long\def\calR{\mathcal{P}}

\global\long\def\frakR{\mathfrak{R}}

\global\long\def\inp#1#2{\left\langle #1,#2\right\rangle }

\global\long\def\eqm#1#2#3{#1\equiv#2\,\left(#3\right)}

\global\long\def\b{\beta}

\global\long\def\pden#1{\mathcal{M}_{#1}(n)}

\global\long\def\lden#1{\mathcal{L}_{#1}(n)}

\global\long\def\Re#1{\mbox{Re}(#1)}

\global\long\def\Res#1#2{\mbox{Res}_{s=#1}#2}

\title{Sums of Two Squares -- Pair Correlation \& Distribution In Short
Intervals }
\begin{abstract}
In this work we show that based on a conjecture for the pair correlation
of integers representable as sums of two squares, which was first
suggested by Connors and Keating and reformulated here, the second
moment of the distribution of the number of representable integers
in short intervals is consistent with a Poissonian distribution, where
``short'' means of length comparable to the mean spacing between
sums of two squares. In addition we present a method for producing
such conjectures through calculations in prime power residue rings
and describe how these conjectures, as well as the above stated result,
may by generalized to other binary quadratic forms. While producing
these pair correlation conjectures we arrive at a surprising result
regarding Mertens' formula for primes in arithmetic progressions,
and in order to test the validity of the conjectures, we present numerical
computations which support our approach.
\end{abstract}
\maketitle

\section{Introduction}

Throughout this work $n,k$ and $h$ will denote positive integers,
$p$ will denote prime numbers and for abbreviation reasons we use
$\eqm abc$ instead of $a\equiv b\mbox{ mod \ensuremath{c}}$. In
addition, we say $m_{p}(n)=k$ if $p^{k}\mid n$ but $p^{k+1}\nmid n$.

\subsection{Background and motivation.}

When studying the distribution of a sequence of integers, for example
the sequence of primes or of those representable as a sum of two squares,
a natural first step would be to understand the mean density of such
integers. For prime numbers this was achieved by Hadamard and de la
Vallée-Poussin with their famous Prime Number Theorem, and for sums
of squares by Landau \cite{key-1}. In order to learn more about
the distribution of such a set the next step would be to look at the
$k-$point correlation, or in other words to find an expression for
\[
\frac{1}{n}\sum_{m=1}^{n}f\left(m+d_{1}\right)\cdots f\left(m+d_{k}\right)
\]
as $n\rightarrow\infty,$ where $f$ is the characteristic function
of the set at hand and $d_{1},...,d_{k}$ are distinct integers. These
correlations give increasingly more precise data about the distribution,
where the $2-$point correlation provides the leading quantitative
estimate of the fluctuations about the mean density of the sequence. 

Regarding the sequence of primes, Hardy and Littlewood gave \cite{key-2}
the following $k-$tuple conjecture for the number $\pi_{\mathbf{d}}(n)$
of positive integers $m\leq n$ for which all of $m+d_{1},...,m+d_{k}$
are prime,\textbf{ $\mathbf{d}=(d_{1},...,d_{k})$} and $d_{1},...,d_{k}$
distinct integers. The conjecture is 
\begin{equation}
\pi_{\mathbf{d}}(n)\sim\mathscr{S}_{\mathbf{d}}\frac{n}{\left(\log n\right)^{k}}\label{eq:Hardy-Littlwood}
\end{equation}
as $n\rightarrow\infty$, provided $\mathscr{S}_{\mathbf{d}}\not=0,$
where the ``singular series'' $\mathscr{S}_{\mathbf{d}}$ is 
\[
\mathscr{S}_{\mathbf{d}}=\prod_{p}\frac{p^{k-1}\left(p-\nu_{\mathbf{d}}(p)\right)}{\left(p-1\right)^{k}}
\]
and $\nu_{\mathbf{d}}(p)$ stands for the number of residue classes
modulo $p$ occupied by $d_{1},...,d_{k}$.

For $k=1$ this is exactly the Prime Number Theorem, and for $k\geq2$
it has not been proved for any $\mathbf{d}$.

\subsection{From a $k-$tuple conjecture to distribution in short intervals.}

We will follow Gallagher's work \cite{key-4} on primes in order to
obtain the moments of the distribution of the number of integers representable
as a sum of two squares in short intervals. Consider first the set
of primes and the Prime Number Theorem, which states that as $n\rightarrow\infty$
\[
\pi(n)\sim\frac{n}{\log n}.
\]

This relation can be understood as the statement that the number of
primes in an interval $\left(m,m+\alpha\right)$, averaged over $m\leq n$,
tends to the limit $\lambda$, when $n$ and $\alpha$ tend to infinity
in such a way that $\alpha\sim\lambda\log n$ with $\lambda$ a positive
constant. 

Gallagher studies the distribution of values of $\pi(m+\alpha)-\pi(m)$
for $m\leq n$ and $\alpha\sim\lambda\log n$, and shows that, assuming
the prime $k-$tuple conjecture of Hardy and Littlewood (\ref{eq:Hardy-Littlwood}),
it suffices that 
\begin{equation}
\sum\limits _{1\leq d_{1},..,d_{k}\leq H}\mathscr{S}_{\mathbf{d}}\sim H^{k}\label{eq:Gallagher's lemma primes}
\end{equation}
as $H\rightarrow\infty$ holds for all $k\in\mathbb{N}$ in order
to prove that all the moments of the distribution tend to moments
of a Poisson distribution, and so the distribution tends to a Poisson
distribution with parameter $\lambda$ as $n\rightarrow\infty$. This
means that the distribution of primes in such intervals is similar
to the distribution of a random set of integers with mean $\lambda$,
and so even though clearly primes are not distributed randomly, in
the perspective of intervals such as those we deal with here they
do. Gallagher has proved (\ref{eq:Gallagher's lemma primes}) in \cite{key-4},
and a simpler proof was presented by Ford \cite{key-10}. We shall
refer to this result as Gallagher's Lemma. 

Consider now the set of integers which are representable as a sum
of two squares and Landau's theorem, which states that $B(n)$, the
number of such integers up to $n$, is given asymptotically by 
\begin{equation}
B(n)\sim\b\frac{n}{\sqrt{\log n}}+O\left(\frac{n}{\log^{\frac{3}{2}}n}\right)\label{eq:Landau's theorem}
\end{equation}
as $n\rightarrow\infty$, where $\b=\frac{1}{\sqrt{2}}\prod\limits _{\eqm p34}\left(1-p^{-2}\right)^{-1/2}$
is the Landau-Ramanujan constant (see \cite{key-13}). 

This relation can be understood as the statement that the number of
integers representable as a sum of two squares in an interval $\left(m,m+\alpha\right)$,
averaged over $m\leq n$, tends to the limit $\lambda$, when $n$
and $\alpha$ tend to infinity in such a way that $\alpha\sim\frac{\lambda}{\b}\sqrt{\log n}$
with $\lambda$ a positive constant. 

We wish to study the distribution of values of $B(m+\alpha)-B(m)$
for $m\leq n$ and $\alpha\sim\frac{\lambda}{\b}\sqrt{\log n}$. In
order to follow Gallagher's method we need first a conjecture analogous
to Hardy and Littlewood's conjecture for sums of two squares, that
is an asymptotic formula for the number $B_{\mathbf{d}}(n)$ of positive
integers $m\leq n$ for which all of $m+d_{1},...,m+d_{k}$ can be
represented as a sum of two squares,\textbf{ $\mathbf{d}=(d_{1},...,d_{k})$}
and $d_{1},...,d_{k}$ distinct integers. The conjecture, analogous
to (\ref{eq:Hardy-Littlwood}), is that there exists a function $\mathscr{T}_{\mathbf{d}}$,
the ``singular series for our problem'', for which the limit 
\begin{eqnarray}
B_{\mathbf{d}}(n)\sim\mathscr{T}_{\mathbf{d}}\frac{n}{\left(\sqrt{\log n}\right)^{k}}\label{eq:Hardy-LittlewoodSumsofsquares}
\end{eqnarray}
holds. If this is so, then the function $\mathscr{T}_{\mathbf{d}}$
depends only on the differences between the $d_{1},...,d_{k}$, in
the sense that $\mathscr{T}_{\mathbf{d}}=\mathscr{T}_{\mathbf{d}+\mathbf{1}}$
where \textbf{$\mathbf{1}=(1,...,1)$}.

Assuming this conjecture, it is enough to show that the singular series
$\mathscr{T}_{\mathbf{d}}$ has mean value $\b$, that is that the
limit 
\begin{equation}
\sum\limits _{1\leq d_{1},..,d_{k}\leq H}\mathscr{T}_{\mathbf{d}}\sim\left(\b H\right)^{k}\label{eq:Gallagher's lemma sum of squares}
\end{equation}
as $H\rightarrow\infty$ holds, for the moments to be Poisson with
parameter $\lambda$.

\subsection{Main result.}

Connors and Keating conjectured in \cite{key-3} that for $k=2$ and
$h=\left|d_{2}-d_{1}\right|$ we have
\[
B_{h}(n)\sim\mathscr{T}_{h}\frac{n}{\left(\sqrt{\log n}\right)^{2}}
\]
as $n\rightarrow\infty$, with the following ``singular series''

\begin{equation}
\mathscr{T}_{h}=2W_{2}(h){\displaystyle \prod_{\begin{array}{c}
\substack{\eqm p34\\
p\mid h
}
\end{array}}}\dfrac{1-p^{-(m_{p}(h)+1)}}{1-p^{-1}}\label{eq:Keating}
\end{equation}
where $m_{p}(h)$ is the power to which the prime $p$ is raised in
the prime decomposition of $h$ and 
\begin{eqnarray*}
W_{2}(h) & = & \begin{cases}
\begin{array}{c}
\frac{1}{4}\\
\frac{2^{m_{2}(h)+1}-3}{2^{m_{2}(h)+2}}
\end{array} & \begin{array}{c}
m_{2}(h)=0\\
m_{2}(h)\geq1
\end{array}\end{cases}
\end{eqnarray*}

\begin{thm}
\label{thm:main result}For $k=2$ and $h=\left|d_{2}-d_{1}\right|$
the singular series $\mathscr{T}_{h}$ has mean value $\b$. More
explicitly 

\[
\sum\limits _{1\leq d_{1\not}\not=d_{2}\leq H}\mathscr{T}_{h}=2\sum_{1\leq h\leq H-1}\left(H-h\right)\mathscr{T}_{h}=\b^{2}H^{2}+O_{\varepsilon}(H^{1+\varepsilon})
\]
as $H\rightarrow\infty$, for all $\varepsilon>0$.
\end{thm}
Assuming the validity of Connors and Keating's pair correlation conjecture
this result implies that Gallagher's Lemma for sums of two squares
and $k=2$ holds, or in other words we show that assuming the conjecture,
the second moment of the distribution of values of $B(m+\alpha)-B(m)$
for $m\leq n$ and $\alpha\sim\frac{\lambda}{\b}\sqrt{\log n}$ is
consistent with a Poissonian distribution with parameter $\lambda$.

\subsection{Mean density and pair correlation.}

We provide a new method of conjecturing estimates for the pair correlation
function stated above, which goes through the mean density and pair
correlation of elements representable as a sum of two squares in residue
rings of the form $\nicefrac{\mathbb{Z}}{p^{k}\mathbb{Z}}$ for primes
$p$ and $k\rightarrow\infty$. The naive expectation for the density
of sums of two squares is 
\begin{equation}
\pden{}:=\frac{1}{2}\prod_{\begin{array}{c}
\substack{\eqm p34\\
p\leq n
}
\end{array}}\left(1+p^{-1}\right)^{-1}\label{eq:Productdensity}
\end{equation}
where $\pden{}$ is simply the product of the densities in the residue
rings described above. We compare this expression with the leading
term of the analytic result for the density of representable integers
given by Landau
\begin{equation}
\lden{}=\frac{\b}{\sqrt{\log n}}\label{eq:Landaudensity}
\end{equation}
and produce the precise ratio between the two and show that
\begin{equation}
y:=\lim_{n\rightarrow\infty}\frac{\pden{}}{\lden{}}=\frac{1}{2}\sqrt{\frac{\pi}{e^{\gamma}}}\label{eq:mitzi constant}
\end{equation}
where $\gamma$ is Euler's constant, using a version of Mertens' formula
in geometric progressions described in \cite{key-6}. Comparing this
to the case of the Prime Number Theorem and Mertens' original formula
\[
\lim_{n\rightarrow\infty}\frac{\prod\limits _{p\leq n}\left(1-p^{-1}\right)}{\pi(n)}=\frac{1}{e^{\gamma}}
\]
we see that as in the case of the primes we are off by a factor. 

Next we derive (\ref{eq:Keating}) in similar methods to those used
for the mean density (\ref{eq:Productdensity}). Denote by $\mathcal{M}^{(2)}(n,h)$
the product of densities of representable pairs $\left(a,a+h\right)$
in the rings $\nicefrac{\mathbb{Z}}{p^{k}\mathbb{Z}}$, $k\rightarrow\infty$
where the product is over primes $p\leq n$ (see Section \ref{sec:Representable-Pairs-In}
for the detailed definition). We then make the following conjecture,
which is equivalent to that of Keating and Connors.
\begin{conjecture}
Let 
\[
Y_{h}(n):=\frac{\mathcal{M}^{(2)}(n,h)}{\frac{1}{n}\#\left\{ \mbox{\ensuremath{m\leq n}: \ensuremath{m}\,\ and \ensuremath{m+h}\,\ are representable}\right\} }.
\]
Then $Y_{h}(n)$ converges and the following limit holds for every
$h\in\mathbb{N}$ 
\[
\lim_{n\rightarrow\infty}Y_{h}(n)=\frac{1}{4}\frac{\pi}{e^{\gamma}}=:y^{2}.
\]

\end{conjecture}
Notice that according to our conjecture the ratio defined above converges
to a universal constant which does not depend on the difference $h$.
In Section \ref{sec:Representable-Pairs-In} we present numeric calculations
to support this conjecture.

\subsection{Generalization to other binary quadratic forms }

Our methods allow us to expand our observation also to integers representable
by other binary quadratic forms $x^{2}+dy^{2}$ with $d=2,3,4,7$
in addition to $d=1$, which are the sums of two squares. The reason
we examine these values of $d$ is that these are the convenient (idoneal)
numbers such that the forms $x^{2}+dy^{2}$ are of class number $1$,
see \cite{key-7} and Definition \ref{Idoneal defenition}. A surprising result
is that the ratio between the product formulas $\pden d$ we present
and the analytic results using variations on Landau's theorem $\lden d$,
for $n\rightarrow\infty$, is in fact constant for the five different
quadratic forms inspected and is again 
\begin{equation}
\lim_{n\rightarrow\infty}\frac{\pden d}{\lden d}=\frac{1}{2}\sqrt{\frac{\pi}{e^{\gamma}}}.\label{eq:ratio}
\end{equation}

We next produce conjectures analogous to (\ref{eq:Keating}) and therefore
to (\ref{eq:Hardy-LittlewoodSumsofsquares}) with $k=2$ for integers
representable by the forms at hand, and finally prove that assuming
our conjectures Gallagher's Lemma holds for $k=2$.

\subsection*{Acknowledgments }

This work is part of the author\textquoteright{}s M. Sc. thesis written
under the supervision of Zeev Rudnick at Tel-Aviv University. Partially
supported by the Israel Science Foundation (grant No. 1083/10). The author would like to thank Zeev Rudnick for his time, patience and guidance, and to the referee for helpful comments on an earlier version of this paper.

\section{Distribution In Short Intervals - Gallagher's Lemma\label{sec:Distribution-In-Small}}

We define $B_{h}(n)$ to be the number of positive integers $m\leq n$
for which both $m$ and $m+h$ can be represented as a sum of two
squares. It is conjectured that $B_{h}(n)\sim\mathscr{T}_{h}\frac{n}{\left(\sqrt{\log n}\right)^{2}}$
where by the Connors and Keating conjecture 
\[
\mathscr{T}_{h}=2W_{2}(h){\displaystyle \prod_{\begin{array}{c}
\substack{\eqm p34\\
p\mid h
}
\end{array}}}\dfrac{1-p^{-(m_{p}(h)+1)}}{1-p^{-1}}
\]
and 
\begin{eqnarray*}
W_{2}(h) & = & \begin{cases}
\begin{array}{c}
\frac{1}{4}\\
\frac{2^{m_{2}(h)+1}-3}{2^{m_{2}(h)+2}}
\end{array} & \begin{array}{c}
m_{2}(h)=0\\
m_{2}(h)\geq1
\end{array}\end{cases}
\end{eqnarray*}

In this section we prove Theorem \ref{thm:main result}, that is we
show that 
\[
\sum\limits _{1\leq d_{1}\not=d_{2}\leq H}\mathscr{T}_{\mathbf{d}}=\sum_{1\leq d_{1}\not=d_{2}\leq H}\mathscr{T}_{\left|d_{2}-d_{1}\right|}=2\sum\limits _{1\leq h\leq H-1}(H-h)\mathscr{T}_{h}=\b^{2}H^{2}+O_{\varepsilon}(H^{1+\varepsilon})
\]
for all $\varepsilon>0$.

Following Gallagher's work for primes described in the introduction,
this calculation will let us obtain the second moment for the distribution
of representable integers in the short intervals described above.

\subsection{Dirichlet series}

Set 
\[
a(h)=2\mathscr{T}_{h}=4W_{2}(h)\prod_{\begin{subarray}{c}
\eqm p34\\
p\mid h
\end{subarray}}\dfrac{1-p^{-(m_{p}(h)+1)}}{1-p^{-1}}.
\]

Notice that $a(h)$ is multiplicative: obviously $a(1)=1$ since 1
is odd and has no prime factors, and for $(m,n)=1$ we have $a(mn)=a(m)a(n)$
because our function is composed of products depending only on the
prime factorizations. 

Computing $a(p^{k})$ gives 
\[
a(p^{k})=\begin{cases}
\begin{array}{cc}
1 & p\equiv1(4)\\
2-\frac{3}{2^{k}} & p=2\\
\frac{1-\frac{1}{p^{k+1}}}{1-\frac{1}{p}} & p\equiv3(4)
\end{array}\end{cases}.
\]
We can thus write 
\begin{eqnarray*}
D(s) & = & \sum\limits _{h=1}^{\infty}a(h)h^{-s}=\prod\limits _{p}\left(1+\sum_{k=1}^{\infty}\frac{a(p^{k})}{p^{ks}}\right)\\
 & = & \left(1+\sum_{k=1}^{\infty}\frac{2-\frac{3}{2^{k}}}{2^{ks}}\right)\prod_{\eqm p14}\left(1+\frac{p^{-s}}{1-p^{-s}}\right)\prod_{\eqm p34}\left(1+\sum_{k=1}^{\infty}\frac{1-\frac{1}{p^{k+1}}}{p^{ks}\left(1-\frac{1}{p}\right)}\right)\\
 & = & R(s)P(s)Q(s)
\end{eqnarray*}
where 
\begin{eqnarray*}
R(s) & = & 1+2\frac{2^{-s}}{1-2^{-s}}-3\frac{2^{-(s+1)}}{1-2^{-(s+1)}}\\
P(s) & = & \prod\limits _{p\equiv1(4)}(1-p^{-s})^{-1}\\
Q(s) & = & \prod_{\eqm p34}\left(1+\frac{1}{1-p^{-1}}\frac{p^{-s}}{1-p^{-s}}-\frac{p^{-1}}{1-p^{-1}}\frac{p^{-(s+1)}}{1-p^{-(s+1)}}\right).
\end{eqnarray*}

\subsection{Comparison to Riemann's $\zeta$ function}

Taking $\zeta(s)=\prod\limits _{p}\left(1-p^{-s}\right)^{-1}$, we
will now show that $\frac{D(s)}{\zeta(s)}$ is analytic for $\sigma>0$
where $s=\sigma+it$, thus $D(s)$ is analytic in that region with
a simple pole at $s=1$. 
\[
\frac{D(s)}{\zeta(s)}=\frac{1+2\frac{2^{-s}}{1-2^{-s}}-3\frac{2^{-(s+1)}}{1-2^{-(s+1)}}}{\left(1-2^{-s}\right)^{-1}}\cdot\prod_{\eqm p34}\frac{1+\frac{1}{1-p^{-1}}\frac{p^{-s}}{1-p^{-s}}-\frac{p^{-1}}{1-p^{-1}}\frac{p^{-(s+1)}}{1-p^{-(s+1)}}}{\left(1-p^{-s}\right)^{-1}}.
\]

The first expression turns out to be 
\[
\frac{R(s)}{\left(1-2^{-s}\right)^{-1}}=1-2^{-s}+2\frac{2^{-s}-2^{-2s}}{1-2^{-s}}-3\frac{2^{-(s+1)}-2^{-(2s+1)}}{1-2^{-(s+1)}}
\]
which is clearly analytic in the desired region. 

The second expression is 
\begin{eqnarray*}
\frac{Q(s)}{\prod\limits _{\eqm p34}\left(1-p^{-s}\right)^{-1}} & = & {\displaystyle \prod_{\eqm p34}}\left(1-p^{-s}+\frac{p^{-s}}{1-p^{-1}}-\frac{p^{-(s+2)}-p^{-(2s+2)}}{\left(1-p^{-1}\right)\left(1-p^{-(s+1)}\right)}\right).
\end{eqnarray*}
Notice that 
\begin{eqnarray*}
1-p^{-s}+\frac{p^{-s}}{1-p^{-1}}-\frac{p^{-(s+2)}-p^{-(2s+2)}}{\left(1-p^{-1}\right)\left(1-p^{-(s+1)}\right)} & = & 1+O\left(\frac{1}{p^{\sigma+1}}\right)
\end{eqnarray*}
and so the product
\[
\frac{Q(s)}{\prod\limits _{\eqm p34}\left(1-p^{-s}\right)^{-1}}={\displaystyle \prod_{\eqm p34}\left(1+O\left(\frac{1}{p^{\sigma+1}}\right)\right)}
\]
converges in the desired region $\sigma>0$ in which it is analytic,
implying that $\frac{D(s)}{\zeta(s)}$ is also analytic there. 

Let $A(s)$ be an analytic function in $\sigma>0$ defined by $D(s)=A(s)\zeta(s)$.
Since $\Res 1{\zeta(s)=1}$, in order to compute $\Res 1{D(s)}$ we
can simply compute $A(1)$ and so 
\[
\Res 1{D(s)}=A(1)=\prod_{\eqm p34}\frac{1}{1-\frac{1}{p^{2}}}=2\b^{2}.
\]

\subsection{Proof of Theorem \ref{thm:main result}}

We write $D(s)=\zeta(s)A(s)$ where $A(s)$ is absolutely convergent
for $\sigma>0$ and so is bounded. We need the following version of
Perron's formula (see for example \cite{key-13}): If 
\[
D(s)=\sum_{n=1}^{\infty}a(h)h^{-s}
\]
 is absolutely convergent for $\mbox{\ensuremath{\sigma}}>1$, then
\[
\sum_{1\leq h\leq H}a(h)(H-h)=\frac{1}{2\pi i}\int\limits _{2-i\infty}^{2+i\infty}\frac{D(s)}{s(s+1)}H^{s+1}ds.
\]

Applying Perron's formula in our case, we are left with evaluating
the contour integral. We want to shift the contour of integration
to $\sigma=\sigma_{0}$, with $0<\sigma_{0}<1$, and so we need to
bound $D(s)$ in this region. First notice that $|A(\sigma+it)|\leq C(\sigma)$
is bounded as it is given by an absolutely convergent product in $\sigma>0$.
In order to bound $\zeta(s)$ we use the classical convexity bound
(see \cite[Chap. II.3]{key-13}) 
\[
|\zeta(\sigma+it)|\ll_{\varepsilon}(1+|t|)^{\frac{1-\sigma}{2}+\varepsilon},\quad0\leq\sigma\leq1,\,\left|t\right|>1.
\]
for all $\varepsilon>0$. Hence the integrand is bounded by 
\begin{equation}
\left|\frac{D(s)}{s(s+1)}H^{s+1}\right|\ll_{\sigma}\begin{cases}
H^{\sigma+1}, & |t|\leq1,\,0<\sigma<1\,\,\,(a)\\
H^{\sigma+1}|t|^{-2+\frac{1-\sigma}{2}+\varepsilon}, & |t|>1,\,0<\sigma\leq1\,\,\,(b)
\end{cases}\label{bound on D}
\end{equation}
and so by shifting contour using bound $(b)$ and picking up a residue
from the simple pole of $\zeta(s)$ at $s=1$ (recall that $\mbox{Res}_{s=1}\zeta(s)=1$)
we have 
\[
\sum_{1\leq h\leq H}a(h)(H-h)=\frac{A(1)}{2}H^{2}+\frac{1}{2\pi i}\int\limits _{\sigma_{0}-i\infty}^{\sigma_{0}+i\infty}\frac{D(s)}{s(s+1)}H^{s+1}ds.
\]

Applying the bounds $(a)$ and $(b)$ in (\ref{bound on D}) allows
us to bound the integral by $O(H^{\sigma_{0}+1})$. In conclusion
we find 
\[
\sum_{1\leq h\leq H}a(h)(H-h)=\beta^{2}H^{2}+O_{\varepsilon}(H^{1+\varepsilon})
\]
for all $\varepsilon>0$. Therefore 
\[
\sum\limits _{1\leq d_{1}\not=d_{2}\leq H}\mathscr{T}_{\mathbf{d}}=\sum_{1\leq h\leq H-1}(H-h)a(h)=\beta^{2}H^{2}+O_{\varepsilon}(H^{1+\varepsilon})
\]
which is effectively Gallagher's Lemma for sums of two squares and
$k=2$.

\section{Sums of Squares in Residue Rings\label{sec: Modulo rings} }

Following Keating and Connors we attempt to produce a $2-$tuple conjecture
using essentially heuristic methods and Landau's theorem. The key
step is to reduce our problem to prime power residue rings, a step
which is made possible by Lemma \ref{lem:An-integer-is} presented
bellow.

\subsection{Representable elements in residue rings}
\begin{prop}
\label{prop:elements in residue rings}Denote by $Sq(p,k)$ the set
of elements representable as a sum of two squares in $\nicefrac{\mathbb{Z}}{p^{k}\mathbb{Z}}$
\[
Sq(p,k)=\left\{ a\in\nicefrac{\mathbb{Z}}{p^{k}\mathbb{Z}}|\, a\mbox{ representable as a sum of two squares}\right\} .
\]
The following holds:

$(a)$ For $\eqm p14$, $Sq(p,k)=\nicefrac{\mathbb{Z}}{p^{k}\mathbb{Z}}$
.

$(b)$ For $\eqm p34$, $Sq(p,k)=\left\{ a\in\nicefrac{\mathbb{Z}}{p^{k}\mathbb{Z}}|\, m_{p}(a)\mbox{ is even or \ensuremath{a=0}}\right\} $.

$(c)$ For $p=2$, $Sq(2,k)=\left\{ a\in\nicefrac{\mathbb{Z}}{2^{k}\mathbb{Z}}|\, a=2^{j}(1+4n),0\leq j\leq k-1\mbox{ or \ensuremath{a=0}}\right\} $.
\end{prop}
Detailed proofs for Proposition \ref{prop:elements in residue rings}
and the other propositions presented in this section can be found
in \cite{key-11}, and they can also be deduced from Lemma A.2 in
\cite{key-12}.
\begin{lem}
\label{lem:An-integer-is}An integer is representable as a sum of
two squares if and only if it is representable as a sum of two squares
in $\nicefrac{\mathbb{Z}}{p^{k}\mathbb{Z}}$ for every prime $p$
and integer $k\in\mathbb{N}$.\end{lem}
\begin{proof}
This is a corollary of Proposition \ref{prop:elements in residue rings}
and of the famous classical result that an integer $n$ is as sum
of two squares if and only if $m_{p}(a)$ is even for all $\eqm p34$
. Say $a=x^{2}+y^{2}$, so obviously $\eqm a{x^{2}+y^{2}}{p^{k}}$.
Conversely assume $a$ is not representable hence $m_{p}(a)$ is odd
for some $\eqm p34$ and so $a$ is not representable in $\nicefrac{\mathbb{Z}}{p^{k}\mathbb{Z}}$
for $k\geq m_{p}(a)$. 
\end{proof}
Equipped with this lemma we shall examine $\nicefrac{\mathbb{Z}}{p^{k}\mathbb{Z}}$
for all primes $p$ and $k\in\mathbb{N}$, and determine which are
the representable elements in these residue rings. This will allow
us to give an expression for the the density of representable elements,
and then of representable pairs.

\subsection{Mean density of representable elements in residue rings.}

We now wish to calculate the densities of representable elements in
$\nicefrac{\mathbb{Z}}{p^{k}\mathbb{Z}}$ for all primes, $k\rightarrow\infty$.
The following propositions provide a method for deriving these limits,
and present ideas which can be useful also for calculating correlations
of higher degrees. 
\begin{prop}
Denote by $Med(p)$ the limit of the mean density of representable
elements in $\nicefrac{\mathbb{Z}}{p^{k}\mathbb{Z}}$ as $k\rightarrow\infty$
\[
Med(p)=\lim_{k\rightarrow\infty}\frac{\#Sq(p,k)}{p^{k}}.
\]
The following holds:

$(a)$ For $\eqm p14$, $Med(p)=1$.

$(b)$ For $\eqm p34$, $Med(p)=\left(1+p^{-1}\right)^{-1}$.

$(c)$ For $p=2$, $Med(2)=\frac{1}{2}$.
\end{prop}

\section{Ratio Between the Product of Densities and Landau's Result\label{sec:Ratio-Between-the-densities}}

\subsection{Density of integers representable as a sum of two squares}

We wish to calculate the mean density of integers representable as
a sum of two squares, so following our approach we take the product
of all the above densities for $p\leq n$: 
\begin{equation}
\pden{}:=\prod_{p\leq n}Med(p)=\frac{1}{2}\prod_{\begin{array}{c}
\substack{\eqm p34\\
p\leq n
}
\end{array}}\left(1+p^{-1}\right)^{-1}.\label{eq:mean density d=00003D1}
\end{equation}

Even though we do not expect that this expression $\pden{}$ will
give us the correct asymptotics, we will show that as for the case
of the primes this Mertens-type product provides the correct answer
up to some constant, and this constant will show a universal property
we will see in Section \ref{sec:Generalization-to-Other}. The leading
term in Landau's analytic expression for the mean density of representable
integers is 
\[
\lden{}=\frac{\b}{\sqrt{\log n}}.
\]

The events that an integer is representable in residue rings associated
with different primes show some dependency, a dependency which gives
rise to a term $y(n)$. Taking this term into consideration we should
have 
\[
\lden{}\sim\frac{\pden{}}{y(n)}.
\]

\subsection{The Ratio }

Mertens' original formula states that 
\[
\prod_{p\leq n}\left(1-p^{-1}\right)=\frac{e^{-\gamma}}{\log n}+O\left(\frac{1}{\log^{2}n}\right)
\]
where $\gamma$ denotes Euler's constant. 

For co-prime integers $a,q$ , Languasco and Zaccagnini show \cite{key-6}
a generalization of Mertens' formula

\small

\begin{eqnarray}
 & {\displaystyle \lim_{n\rightarrow\infty}\left(\log n\right)^{1/\varphi(q)}\prod_{\begin{array}{c}
\substack{\eqm paq\\
p\leq n
}
\end{array}}\left(1-p^{-1}\right)}={\displaystyle \left[e^{-\gamma}\prod_{p}\left(1-p^{-1}\right)^{\alpha(p;a,q)}\right]^{1/\varphi(q)}}\label{eq:Mertens}
\end{eqnarray}
 \normalsize where $\varphi$ is Euler's totient function, and $a(p;a,q)$
is given by
\[
a(p;a,q)=\begin{cases}
\varphi(q)-1 & ,\,\eqm paq\\
-1 & ,\,\mbox{otherwise}
\end{cases}
\]

\begin{thm}
Let $y(n)=\frac{\pden{}}{\lden{}}$ be the ratio between the product
of densities in prime power residue rings and Landau's leading term.
Then $y(n)$ converges as $n$ tends to infinity and the limit is
given by 
\[
y:=\lim_{n\rightarrow\infty}y(n)=\lim_{n\rightarrow\infty}\frac{\pden{}}{\lden{}}=\frac{1}{2}\sqrt{\frac{\pi}{e^{\gamma}}}.
\]
\end{thm}
\begin{proof}
Plugging $a=3,\, q=4$ in Mertens' formula for primes in arithmetic
progression we have 
\[
\prod_{\begin{array}{c}
\substack{\eqm p34\\
p\leq n
}
\end{array}}\left(1-p^{-1}\right)\sim\frac{e^{-\gamma/2}}{\sqrt{\log n}}\left[\frac{\prod\limits _{\eqm p34}\left(1-p^{-1}\right)}{\left(1-2^{-1}\right)\prod\limits _{\eqm p14}\left(1-p^{-1}\right)}\right]^{1/2}
\]
and since 
\[
\left(1+p^{-1}\right)^{-1}=\frac{1-p^{-1}}{1-p^{-2}}
\]
we arrive at

\small
\[
\prod_{\begin{array}{c}
\substack{\eqm p34\\
p\leq n
}
\end{array}}\left(1+p^{-1}\right)^{-1}\sim\frac{\sqrt{2}e^{-\gamma/2}}{\sqrt{\log n}}\prod\limits _{\eqm p34}\left(1-p^{-1}\right)^{-\frac{1}{2}}\left(1+p^{-1}\right)^{-1}\prod\limits _{\eqm p14}\left(1-p^{-1}\right)^{-\frac{1}{2}}.
\]

\normalsize

We are interested in the ratio 
\[
\lim_{n\rightarrow\infty}\frac{\pden{}}{\lden{}}=\lim_{n\rightarrow\infty}\frac{\pden{}}{\b/\sqrt{\log n}}
\]
with $\b$ the Landau-Ramanujan constant given by 
\[
\b=\frac{1}{\sqrt{2}}\prod\limits _{\eqm p34}\left(1-p^{-2}\right)^{-1/2}
\]
and so 
\[
\lim_{n\rightarrow\infty}\frac{\pden{}}{\lden{}}=\frac{1}{2}\cdot2e^{-\gamma/2}\prod\limits _{\eqm p34}\left(1+p^{-1}\right)^{-1/2}\prod\limits _{\eqm p14}\left(1-p^{-1}\right)^{-1/2}.
\]

The two products are exactly $\sqrt{L(1)}$ which is calculated in
\cite{key-9}, where $L(s)$ is the Dirichlet series for the non principal
character modulo $4$. Therefore 
\[
y=\lim_{n\rightarrow\infty}\frac{\pden{}}{\lden{}}=e^{-\gamma/2}\sqrt{\frac{\pi}{4}}=\frac{1}{2}\sqrt{\frac{\pi}{e^{\gamma}}}.
\]

\end{proof}
This is quite an elegant result, which can be easily generalized using
similar tools as will be done in section 6.

\section{Representable Pairs In Residue Rings And Their Densities\label{sec:Representable-Pairs-In}}

We are now in a position to look at the distribution of representable
pairs $a,a+h$ in $\nicefrac{\mathbb{Z}}{p^{k}\mathbb{Z}}$. The following
proposition states the densities of representable pairs, for a detailed
proof see \cite{key-11}.
\begin{prop}
Denote by $Med^{(2)}(p,h)$ the limit of the mean density of representable
pairs $\left(a,a+h\right)$ in $\nicefrac{\mathbb{Z}}{p^{k}\mathbb{Z}}$
as $k\rightarrow\infty$ 
\[
Med^{(2)}(p,h)=\lim_{k\rightarrow\infty}\frac{\#\left\{ a,a+h\in Sq(p,k)\right\} }{p^{k}}.
\]
The following holds:

$(a)$ For $\eqm p14$, $Med^{(2)}(p,h)=1$.

$(b)$ For $\eqm p34$, $Med^{(2)}(p,h)=\dfrac{1-p^{-(m_{p}(h)+1)}}{1+p^{-1}}$
.

$(c)$ For $p=2$, $Med^{(2)}(2,h)=W_{2}\left(h\right)=\begin{cases}
\begin{array}{c}
\frac{1}{4}\\
\frac{2^{m_{2}(h)+1}-3}{2^{m_{2}(h)+2}}
\end{array} & \begin{array}{c}
m_{2}(h)=0\\
m_{2}(h)\geq1
\end{array}\end{cases}$ .
\end{prop}
We define $\mathcal{M}^{(2)}(n,h)$ to be the product of the above
densities 
\[
\mathcal{M}^{(2)}(n,h)=\prod_{p\leq n}Med^{(2)}(p,h)=W_{2}(h)\prod_{\begin{array}{c}
\substack{\eqm p34\\
p\leq n
}
\end{array}}\dfrac{1-p^{-(m_{p}(h)+1)}}{1+p^{-1}}.
\]
And as in the introduction we denote

\[
B_{h}(n)=\#\left\{ m\leq n|\, m,m+h\,\mbox{are representable as a sum of two squares}\right\} 
\]
and define 
\[
Y_{h}(n):=\frac{\mathcal{M}^{(2)}(n,h)}{\frac{1}{n}B_{h}(n)}.
\]
The density of representable pairs is thus:
\[
\frac{1}{n}B_{h}(n)\sim\frac{1}{Y_{h}(n)}\cdot W_{2}(h)\prod_{\begin{array}{c}
\substack{\eqm p34\\
p\leq n
}
\end{array}}\dfrac{1-p^{-(m_{p}(h)+1)}}{1+p^{-1}}.
\]

We extract the asymptotic term depending on $n$ from the above expression
using the ratio computed in Section \ref{sec:Ratio-Between-the-densities}.
Recall 
\[
\lden{}=\frac{\beta}{\sqrt{\log n}}=\frac{\prod\limits _{\eqm p34}\left(1-p^{-2}\right)^{-\frac{1}{2}}}{\sqrt{2\log n}}
\]
and so from Landau's Theorem together with the previous sections we
have for all $n$ 
\[
1=\left(\frac{\lden{}}{\pden{}/y(n)}\right)^{2}=\frac{2y(n)^{2}}{\log n}\prod_{\begin{array}{c}
\substack{\eqm p34\\
p\leq n
}
\end{array}}\dfrac{\left(1-p^{-2}\right)^{-1}}{\left(1+p^{-1}\right)^{-2}}
\]
and so we write 
\begin{eqnarray*}
 & \frac{1}{n}B_{h}(n)\sim\frac{1}{Y_{h}(n)}\cdot W_{2}(h){\displaystyle \prod_{\begin{array}{c}
\substack{\eqm p34\\
p\leq n
}
\end{array}}}\dfrac{1-p^{-(m_{p}(h)+1)}}{1+p^{-1}}\cdot\left(\frac{\lden{}}{\pden{}/y(n)}\right)^{2} & .
\end{eqnarray*}
Since
\begin{eqnarray*}
\dfrac{\left(1-p^{-(m_{p}(h)+1)}\right)}{\left(1+p^{-1}\right)}\frac{\left(1-p^{-2}\right)^{-1}}{\left(1+p^{-1}\right)^{-2}} & = & \dfrac{1-p^{-(m_{p}(h)+1)}}{1-p^{-1}}
\end{eqnarray*}
we have

\begin{eqnarray}
\frac{1}{n}B_{h}(n) & \sim & \frac{1}{\log n}\cdot\left(\frac{y(n)^{2}}{Y_{h}(n)}\right)\cdot2W_{2}(h){\displaystyle \prod_{\begin{array}{c}
\substack{\eqm p34\\
p\leq n
}
\end{array}}}\dfrac{1-p^{-(m_{p}(h)+1)}}{1-p^{-1}}\label{eq:pair normalization}
\end{eqnarray}

For $p$ such that $m_{p}(h)=0$ the product is $1$, and since we
are interested in $n\rightarrow\infty$, we can assume $n\geq h$
and so the product is over all $\eqm p34$ such that $p\mid h$. The
conjecture presented by Connors and Keating is thus equivalent to
the conjecture that for all $h$ 
\[
\frac{y(n)^{2}}{Y_{h}(n)}\rightarrow1
\]
as $n\rightarrow\infty$, which can be also stated as 
\[
\lim_{n\rightarrow\infty}Y_{h}(n)=\frac{1}{4}\frac{\pi}{e^{\gamma}}=y^{2}\mbox{ for all \ensuremath{h\in\mathbb{N}}},
\]
a conjecture for which we present numerical computations in Section
\ref{sec:Numerics}. Another interpretation of this conjecture would
be that as in the case of the density of integers representable as
a sum of two squares, the product expression $\mathcal{M}^{2}(n,h)$
gives the correct estimate up to a constant.

Assuming the validity of this conjecture the density of representable
pairs is given by 
\[
\frac{1}{n}B_{h}(n)=\frac{1}{\log n}\cdot2W_{2}(h)\prod_{\begin{array}{c}
\substack{\eqm p34\\
p\mid h
}
\end{array}}\dfrac{1-p^{-(m_{p}(h)+1)}}{1-p^{-1}}
\]
and so 
\[
\mathscr{T}_{h}=2W_{2}(h)\prod_{\begin{array}{c}
\substack{\eqm p34\\
p\mid h
}
\end{array}}\dfrac{1-p^{-(m_{p}(h)+1)}}{1-p^{-1}}.
\]

\section{Generalization to Other Binary Quadratic Forms\label{sec:Generalization-to-Other}}

In this section we generalize our conjectures and results for additional
binary quadratic forms.

\subsection{Preliminaries}

Let us look at the following family of positive definite binary quadratic
forms 
\[
q(d;x,y)=x^{2}+dy^{2}.
\]

\begin{defn}
\label{Idoneal defenition}We say that $d\in\mathbb{N}$ is a convenient
(idoneal) number if there is finite set of primes $S$, an integer
$N$ and congruence classes $c_{1},...,c_{k}\mbox{ mod \ensuremath{N}}$
such that for all primes $p\not\in S$
\[
p=x^{2}+dy^{2}\iff\eqm p{c_{1},...,c_{k}}N.
\]
\end{defn}
\begin{example}
For $d=1$, $S=\left\{ 2\right\} $, $c_{1}=1$ and $N=4$ we have
Fermat's result for sums of two squares. 
\end{example}
We focus here on convenient $d$'s such that the form $x^{2}+dy^{2}$
is of class number $1$ which are $d=1,2,3,4,7$ . In these cases
one can fully determine if an integer $n$ is representable by the
form simply by making sure that the primes which are not representable
appear with an even multiplicity in the integer's prime factorization. 

Again we are first interested in the mean density of representable
integers, and we can calculate the densities in the residue rings
in the exact same way that we did for $d=1$ and thus generalize (\ref{eq:mean density d=00003D1})
. In \cite{key-8} Shanks produces Landau's constants $\b_{1},\b_{2},\b_{3,}\b_{4},\b_{7}$
for which 
\[
B(d,n):=\#\left\{ m\leq n|\,\mbox{\ensuremath{m}is of the form \ensuremath{x^{2}+dy^{2}}}\right\} \sim\b_{d}\frac{n}{\sqrt{\log n}}
\]
as $n\rightarrow\infty$, and so we can again calculate the ratio
between the product and the analytic expressions as was done in Section
\ref{sec:Ratio-Between-the-densities} for sums of squares. 

First let us recall the following classical results (see \cite{key-7}):
\begin{thm}
\label{thm:integers representable by d-forms}An integer $n$ is representable
by the form \textup{$x^{2}+dy^{2}$ if and only if:}\end{thm}
\begin{itemize}
\item If $d=1$, $m_{p}(n)$ is even for all primes $\eqm p34$.
\item If $d=2$, $m_{p}(n)$ is even for all primes $\eqm p{5,7}8$.
\item If $d=3$, $m_{p}(n)$ is even for all primes $\eqm p23$.
\item If $d=4$, $m_{p}(n)$ is even for all primes $\eqm p34$ and $m_{2}(n)\not=1$
.
\item If $d=7$, $m_{p}(n)$ is even for all primes $\eqm p{3,5,6}7$ and
$m_{2}(n)\not=1$.
\end{itemize}
The conditions for representation by these forms bare obvious resemblance.
\begin{defn}
For convenience reasons we divide the primes into the following sets:\end{defn}
\begin{itemize}
\item Say $p\in Q_{d}$ if $p$ is a prime such that $n=x^{2}+dy^{2}\Rightarrow m_{p}(n)$
is even. Notice that by Theorem \ref{thm:integers representable by d-forms}
and by the Prime Number Theorem for Arithmetic Progressions this set
consists of approximately half of the primes. 
\item Say $p\in R_{d}$ if $p$ is a prime such that $p\not\in Q_{d}$ and
$n=x^{2}+dy^{2}\Rightarrow m_{p}(n)$ has some constraint as described
in Theorem \ref{thm:integers representable by d-forms}, or if $p$
is such that $(p,N)\not=1$ where $N$ is as described in Definition
\ref{Idoneal defenition}, that is $N$ such that $p=x^{2}+dy^{2}\iff\eqm p{c_{1},...,c_{k}}N$.
Notice that $R_{d}$ is a finite set. 
\item say $p\in P_{d}$ if $p$ is a prime such that $p\not\in Q_{d}\bigcup R_{d}$,
or more directly if $p$ is of the form $x^{2}+dy^{2}$ and $(p,N)=1$.
Again this set consists of approximately half of the primes. \end{itemize}
\begin{example}
For the case of sums of squares, that is $d=1$, we write
\begin{eqnarray*}
 & Q_{1}=\left\{ p\mbox{ prime\ensuremath{:\,\eqm p34}}\right\} ,P_{1}=\left\{ p\mbox{ prime\ensuremath{:\,\eqm p14}}\right\} ,\\
 & R_{1}=\left\{ 2\right\} 
\end{eqnarray*}
For $d=7$ we write 
\begin{eqnarray*}
 & Q_{7}=\left\{ p\mbox{ prime\ensuremath{:\,\eqm p{3,5,6}7}}\right\} ,P_{7}=\left\{ p\mbox{ prime\ensuremath{:\,\eqm p{1,2,4}7}, \ensuremath{p\not=2}}\right\} ,\\
 & R_{7}=\left\{ 2,7\right\} .
\end{eqnarray*}

\end{example}
It is important to note that the reason we define the sets of primes
$Q_{d},R_{d},P_{d}$ the way we do and not by the values of $\left(\frac{-d}{p}\right)$,
which stands for the Legendre symbol, is that the Legendre symbol
is only defined for odd primes $p$, while the prime $p=2$ plays
an important role in our computations. On the other hand it will be
useful for us to notice that for $d=1,2,3,4,7$ indeed 
\[
\left(\frac{-d}{p}\right)=1\iff p\in P_{d}
\]
and 
\[
\left(\frac{-d}{p}\right)=-1\iff p\in Q_{d}
\]
unless $d=3$, in which case $Q_{3}=\left\{ p:\,\left(\frac{-d}{p}\right)=-1\right\} \cup\left\{ 2\right\} $.

\subsection{Ratio between the product density and Landau\textquoteright{}s density}

We continue by following the same methods established in Sections
\ref{sec: Modulo rings} and \ref{sec:Ratio-Between-the-densities}
for the definition of $\pden{}$ in order to define a product expression
$\pden d$ associated with the mean density of integers of the form
$x^{2}+dy^{2}$. Notice that for all the above $d$'s the condition
for being representable by the form is over the primes in $Q_{d}$,
plus some local conditions over the primes in $R_{d}$. Similarly
to what we have done in the previous sections we define the naive
expectation of the density of integers representable by the form $x^{2}+dy^{2}$
as 
\[
\pden d=\prod_{p\in R_{d}}w_{d}(p)\prod_{\begin{subarray}{c}
p\in Q_{d}\\
p\leq n
\end{subarray}}\left(1+p^{-1}\right)^{-1}
\]
with $w_{d}(p)$ the mean density of representable element in $\nicefrac{\mathbb{Z}}{p^{\text{k}}\mathbb{Z}}$,
$k\rightarrow\infty$, for $p\in R_{d}$. The primes $p\in P_{d}$
do not participate here since similarly to the case of sums of two
squares, the mean density of representable elements in $\nicefrac{\mathbb{Z}}{p^{\text{k}}\mathbb{Z}}$,
$k\rightarrow\infty$, is $1$. 

These products can be computed using Mertens' formula for arithmetic
progressions, as was done in the previous section for $d=1$:
\begin{eqnarray*}
\prod_{\substack{p\in Q_{d}\\
p\leq n
}
}\left(1+p^{-1}\right)^{-1} & \sim & \frac{e^{-\gamma/2}}{\sqrt{\log n}}\prod\limits _{p\in Q_{d}}\left(1-p^{-1}\right)^{-\frac{1}{2}}\left(1+p^{-1}\right)^{-1}\prod\limits _{p\in P_{d}\cup R_{d}}\left(1-p^{-1}\right)^{-\frac{1}{2}}.
\end{eqnarray*}

Again we are interested in the analogue of (\ref{eq:ratio}), that
is in the ratio between these products and the leading term of the
analytic expression given by the generalization of Landau's theorem
as shown in \cite{key-8}:

\begin{equation}
\lden d=\frac{\b_{u}}{\sqrt{\log n}}\,\,,\,\,\,\,\b_{d}=\delta_{d}\cdot g_{d}\cdot\left(\frac{L_{d}(1)\cdot2\left|d\right|}{\pi\varphi(2\left|d\right|)}\right)^{\frac{1}{2}}\label{eq:LanduaGeneralized}
\end{equation}
with $\varphi$ the Euler totient function and 
\begin{eqnarray*}
g_{d} & = & \prod_{\left(\frac{-d}{p}\right)=-1}\left(1-p^{-2}\right)^{-\frac{1}{2}}\\
L_{d}(s) & = & \sum_{\mbox{odd \ensuremath{n}}}\left(\frac{-d}{n}\right)n^{-s}=\prod_{\left(\frac{-d}{p}\right)=1}\left(1-p^{-s}\right)^{-1}\prod_{\left(\frac{-d}{p}\right)=-1}\left(1+p^{-s}\right)^{-1}\\
\delta_{d} & = & \begin{cases}
1 & ,\, d=1,2\\
\frac{2}{3} & ,\, d=3\\
\frac{3}{4} & ,\, d=4,7
\end{cases}
\end{eqnarray*}
Reformulating the products above we have
\[
g_{d}=\prod_{2\not=p\in Q_{d}}\left(1-p^{-2}\right)^{-\frac{1}{2}}=\gamma_{d}\prod_{p\in Q_{d}}\left(1-p^{-2}\right)^{-\frac{1}{2}}
\]
where $\gamma_{d}=\begin{cases}
1 & ,\, d=1,2,4,7\\
\frac{\sqrt{3}}{2} & ,\, d=3
\end{cases}$, and

\begin{eqnarray*}
\sqrt{L_{d}(1)} & = & \prod_{p\in P_{d}}\left(1-p^{-1}\right)^{-\frac{1}{2}}\prod_{2\not=p\in Q_{d}}\left(1+p^{-1}\right)^{-\frac{1}{2}}\\
 & = & \prod_{p\in P_{d}}\left(1-p^{-1}\right)^{-\frac{1}{2}}\lambda_{d}\prod_{p\in Q_{d}}\left(1+p^{-1}\right)^{-\frac{1}{2}}
\end{eqnarray*}
where $\lambda_{d}=\begin{cases}
1 & ,\, d=1,2,4,7\\
\sqrt{\frac{3}{2}} & ,\, d=3
\end{cases}$. \\
The ratio in question is therefore given by
\begin{eqnarray*}
 & {\displaystyle \lim_{n\rightarrow\infty}y_{d}(n)=}{\displaystyle \lim_{n\rightarrow\infty}\frac{\pden d}{\lden d}=\lim_{n\rightarrow\infty}\frac{\pden d}{\b_{d}/\sqrt{\log n}}}=\\
 & {\displaystyle \prod_{\begin{subarray}{c}
p\in R_{d}\end{subarray}}}w_{d}(p)\frac{1}{\delta_{d}}\sqrt{\frac{\pi}{e^{\gamma}}}\cdot\sqrt{\frac{\varphi(2\left|d\right|)}{2\left|d\right|}}\frac{1}{\gamma_{d}\lambda_{d}}{\displaystyle \prod_{\begin{subarray}{c}
p\in R_{d}\end{subarray}}}\left(1-p^{-1}\right)^{-\frac{1}{2}}.
\end{eqnarray*}

Recall $\frac{\varphi(n)}{n}=\prod\limits _{p|n}\left(1-p^{-1}\right)$.
For $d=1,2,4,7$ we have $p|2d\iff p\in R_{d}$ and so the products
cancel each other. For $d=3$ we have $2|2d$ and $2\not\in R_{3}$,
so we are left with the term $\left(1-2^{-1}\right)^{\frac{1}{2}}=\frac{1}{\sqrt{2}}$.
Since $\frac{1}{\sqrt{2}}\frac{2}{\sqrt{3}}\frac{\sqrt{2}}{\sqrt{3}}=\frac{2}{3}$
we can write 

\[
\lim_{n\rightarrow\infty}y_{d}(n)=\prod_{p\in R_{d}}w_{d}(p)\frac{1}{\delta_{d}}\sqrt{\frac{\pi}{e^{\gamma}}}\cdot s_{d}
\]
where $s_{d}=\begin{cases}
1 & ,\, d=1,2,4,7\\
\frac{2}{3} & ,\, d=3
\end{cases}$.

Computing case by case we prove the following theorem:
\begin{thm}
For $d=1,2,3,4,7$ the ratio between the product of densities in the
residue rings and Landau's density of integers representable by the
forms $x^{2}+dy$ converges to $\frac{1}{2}\sqrt{\frac{\pi}{e^{\gamma}}}$
as $n\rightarrow\infty$, that is 
\[
\lim_{n\rightarrow\infty}y_{d}(n)=\lim_{n\rightarrow\infty}\frac{\pden d}{\lden d}=\frac{1}{2}\sqrt{\frac{\pi}{e^{\gamma}}}=y.
\]

\end{thm}
This is quite a surprising result, which makes the constant $y=\frac{1}{2}\sqrt{\frac{\pi}{e^{\gamma}}}$
somewhat universal as the ratio between the density of integers representable
by the forms at hand and the naively constructed Mertens-type products
we have presented.

\subsection{Pair correlation conjecture }

We can now propose a conjecture for the pair correlation function
for the forms $x^{2}+dy^{2}$ with $d=1,2,3,4,7$, generalizing (\ref{eq:Keating})
and (\ref{eq:Hardy-LittlewoodSumsofsquares}) . Denote by $W_{d,p}(h)$
the density of representable pairs $\left(a,a+h\right)$ in $\nicefrac{\mathbb{Z}}{p^{\text{k}}\mathbb{Z}}$,
$k\rightarrow\infty$, for $p\in R_{d}$, and $Y_{d,h}(n)$ the dependance
term which must be taken into consideration. We extract the asymptotic
term depending on $n$ exactly as was done in Section \ref{sec:Representable-Pairs-In}:

\begin{eqnarray*}
 & {\displaystyle \frac{1}{Y_{d,h}(n)}\prod_{p\in R_{d}}W_{d,p}(h)\cdot\prod_{\begin{subarray}{c}
p\in Q_{d}\\
p\leq n
\end{subarray}}\dfrac{1-p^{-(m_{p}(h)+1)}}{1+p^{-1}}}\\
\sim & {\displaystyle \frac{1}{Y_{d,h}(n)}\prod_{p\in R_{d}}W_{d,p}(h)\prod_{\begin{subarray}{c}
p\in Q_{d}\\
p\leq n
\end{subarray}}\dfrac{1-p^{-(m_{p}(h)+1)}}{1+p^{-1}}\left(\frac{\lden d}{\pden d/y_{d}(n)}\right)^{2}}\\
\sim & {\displaystyle \frac{1}{\log n}\cdot\left(\frac{y_{d}^{2}(n)}{Y_{d,h}(n)}\right)\prod_{p\in R_{d}}\frac{W_{d,p}(h)}{w_{d}^{2}(p)}\prod_{\begin{subarray}{c}
p\in Q_{d}\\
p\leq n
\end{subarray}}\dfrac{1-p^{-(m_{p}(h)+1)}}{\left(1+p^{-1}\right)^{-1}}\cdot\delta_{d}^{2}\cdot g_{d}^{2}\cdot\frac{L_{d}(1)\cdot2\left|d\right|}{\pi\varphi(2\left|d\right|)}} & .
\end{eqnarray*}

Let us first look at the products at hand. As before
\begin{eqnarray*}
 & {\displaystyle \prod_{\begin{subarray}{c}
p\in Q_{d}\\
p\leq n
\end{subarray}}\dfrac{1-p^{-(m_{p}(h)+1)}}{\left(1+p^{-1}\right)^{-1}}\cdot g_{d}^{2}=\prod_{\begin{subarray}{c}
p\in Q_{d}\\
p\leq n
\end{subarray}}\dfrac{1-p^{-(m_{p}(h)+1)}}{\left(1+p^{-1}\right)^{-1}}\prod_{2\not=p\in Q_{d}}\left(1-p^{-2}\right)^{-1}}\\
 & {\displaystyle \sim\prod_{\begin{subarray}{c}
p\in Q_{d}\\
p\mid h
\end{subarray}}\dfrac{1-p^{-(m_{p}(h)+1)}}{1-p^{-1}}\cdot S_{d}}
\end{eqnarray*}
where $S_{d}=\begin{cases}
1 & ,\, d=1,2,4,7\\
\frac{3}{4} & ,\, d=3
\end{cases}$.

Again the conjecture is that for all $h$ 
\[
\frac{y_{d}^{2}(n)}{Y_{d,h}(n)}\rightarrow1
\]
as $n\rightarrow\infty$, and so assuming the validity of this conjecture
the density of pairs of the form $x^{2}+dy^{2}$ is given by 
\[
\frac{1}{n}B_{h}(d,n)=\frac{1}{\log n}\cdot\mathscr{T}_{d,h}
\]
where 
\begin{equation}
\mathscr{T}_{d,h}=c_{d}\prod_{p\in R_{d}}W_{d,p}(h)\prod_{\begin{subarray}{c}
p\in Q_{d}\\
p\mid h
\end{subarray}}\dfrac{1-p^{-(m_{p}(h)+1)}}{1-p^{-1}}\label{eq:d-pair correlation}
\end{equation}
and 
\[
c_{d}=\delta_{u}^{2}\frac{L_{u}(1)\cdot2\left|u\right|}{\pi\varphi(2\left|u\right|)}\prod\limits _{p\in R_{d}}\frac{1}{w_{d}^{2}(p)}S_{d}
\]

It is left to compute $c_{d}$ and $W_{d,p}(h)$ case by case. Dirichlet's
class number formula (see \cite{key-9}) gives 
\begin{equation}
L_{1}(1)=\frac{\pi}{4},\, L_{2}(1)=\frac{\pi}{2\sqrt{2}},\, L_{3}(1)=\frac{\pi}{2\sqrt{3}},\, L_{4}(1)=\frac{\pi}{4},\, L_{7}(1)=\frac{\pi}{2\sqrt{7}}\label{eq:L-function}
\end{equation}
and so plugging all the different term we have
\begin{equation}
c_{1}=2,\, c_{2}=2\sqrt{2},\, c_{3}=\frac{2}{\sqrt{3}},\, c_{4}=2,\, c_{7}=\frac{2\sqrt{7}}{3}.\label{eq:c-constants}
\end{equation}
In addition, calculations similar to those shown for sums of squares
in Section \ref{sec:Representable-Pairs-In} give 
\[
W_{1,2}(h)=\begin{cases}
\dfrac{1}{4} & ,\, m_{2}(h)=0\\
\dfrac{2^{m_{2}(h)+1}-3}{2^{m_{2}(h)+2}} & ,\, m_{2}(h)\geq1
\end{cases},\,\, W_{2,2}(h)=\begin{cases}
\dfrac{1}{4} & ,\, m_{2}(h)=0,1\\
\dfrac{2^{m_{2}(h)}-3}{2^{m_{2}(h)+1}} & ,\, m_{2}(h)\geq2
\end{cases}
\]

\[
W_{3,3}(h)=\frac{1}{2}\cdot\dfrac{3^{m_{3}(h)+1}-2}{3^{m_{3}(h)+1}},\,\,\,\,\,\,\,\,\,\,\,\,\,\,\,\, W_{4,2}(h)=\begin{cases}
\frac{1}{8} & ,\, m_{2}(h)=0\\
0 & ,\, m_{2}(h)=1\\
\frac{5}{16} & ,\, m_{2}(h)=2\\
\frac{3\cdot2^{m_{2}(h)-1}-3}{2^{m_{2}(h)+2}} & ,\, m_{2}(h)\geq3
\end{cases}
\]

\[
W_{7,2}(h)=\begin{cases}
\frac{1}{2} & ,\, m_{2}(h)=0,1\\
\frac{3}{4} & ,\, m_{2}(h)\geq2
\end{cases},\,\,\,\,\,\,\,\,\,\,\,\,\,\,\,\,\,\,\,\,\,\,\,\,\,\,\,\,\,\,\, W_{7,7}(h)=\frac{1}{2}\cdot\dfrac{7^{m_{7}(h)+1}-4}{7^{m_{7}(h)+1}}.
\]

\subsection{Distribution in short intervals - the second moment}

We wish to generalize our result from Section \ref{sec:Distribution-In-Small}
concerning the second moments of the distribution of representable
integers in short intervals. 

We are interested in the distribution of values of $B(d,m+\alpha_{d})-B(d,m)$,
which stands for the number of integers of the form $x^{2}+dy^{2}$
in the interval $\left(m,m+\alpha_{d}\right)$, for $m\leq n$ and
$\alpha_{d}\sim\frac{\lambda}{\b_{d}}\sqrt{\log n}$ . Assuming (\ref{eq:d-pair correlation})
we wish to show that the second moment of this distribution is consistent
with a Poissonian distribution with parameter $\lambda$, and so we
prove Gallagher's Lemma for integers of the form $x^{2}+dy^{2},\, d=1,2,3,4,7$
and $k=2$.
\begin{thm}
The singular series $\mathscr{T}_{d,h}$ has mean value $\b_{d}$
for $d=1,2,3,4,7$ as defined in (\ref{eq:LanduaGeneralized}). More
explicitly 

\[
\sum\limits _{1\leq d_{1\not}\not=d_{2}\leq H}\mathscr{T}_{d,h}=2\sum_{1\leq h\leq H-1}\left(H-h\right)\mathscr{T}_{d,h}=\b_{d}^{2}H^{2}+O_{\varepsilon}(H^{1+\varepsilon})
\]
as $H\rightarrow\infty$, for all $\varepsilon>0$.\end{thm}
\begin{proof}
We follow the proof described in Section \ref{sec:Distribution-In-Small}.
First we normalize $\mathscr{T}_{d,h}$ by defining $a_{d}(h)=\frac{\mathscr{T}_{d,h}}{\mathscr{T}_{d,1}}$,
which is now multiplicative. Then we show that the corresponding Dirichlet
series $D_{d}(s)$ has a simple pole at $s=1$ with residue $\frac{\b_{d}}{\mathscr{T}_{d,1}}$,
as was done for sums of squares. Following the exact same steps detailed
in Section \ref{sec:Distribution-In-Small} we have 
\[
D_{d}(s)=R_{d}(s)P_{d}(s)Q_{d}(s)
\]
where

\begin{eqnarray*}
R_{d}(s) & = & \prod\limits _{p\in R_{d}}\left(1+\sum_{k=1}^{\infty}\frac{a_{d}(p^{k})}{p^{ks}}\right)\\
P_{d}(s) & = & \prod_{p\in P_{d}}\left(1-p^{-s}\right)^{-1}\\
Q_{d}(s) & = & \prod_{p\in Q_{d}}\left(1+\frac{1}{1-p^{-1}}\frac{p^{-s}}{1-p^{-s}}-\frac{p^{-1}}{1-p^{-1}}\frac{p^{-(s+1)}}{1-p^{-(s+1)}}\right).
\end{eqnarray*}

It can be shown $D_{d}(s)=A_{d}(s)\zeta(s)$ with $A_{d}(s)$ analytic,
and so to calculate the residue of $D(s)$ at $s=1$ it is left to
calculate $A_{d}(1)$ which gives 
\[
A_{d}(1)=\lim_{s\rightarrow1}\frac{D_{d}(s)}{\zeta(s)}=\prod\limits _{p\in R_{d}}\frac{1+\sum\limits _{k=1}^{\infty}\frac{a_{d}(p^{k})}{p^{k}}}{\left(1-p^{-1}\right)^{-1}}\prod_{p\in Q_{d}}\left(1-p^{-2}\right)^{-1}.
\]

Recall 
\[
\b_{d}^{2}=\delta_{d}^{2}\frac{L_{d}(1)\cdot2\left|d\right|}{\pi\varphi(2\left|d\right|)}\prod_{p\in Q_{d}}\left(1-p^{-2}\right)^{-1}
\]
and so it remains to show that indeed for $d=1,2,4,7$
\[
\prod\limits _{p\in R_{d}}\frac{1+\sum\limits _{k=1}^{\infty}\frac{a_{d}(p^{k})}{p^{k}}}{\left(1-p^{-1}\right)^{-1}}=\frac{\delta_{d}^{2}\frac{L_{d}(1)\cdot2\left|d\right|}{\pi\varphi(2\left|d\right|)}}{\mathscr{T}_{d,1}}
\]

We continue exactly as was detailed in Section \ref{sec:Distribution-In-Small}
for sums of squares. Plugging in all the relevant constants, all computed
above, we arrive at the desired result.
\end{proof}

\section{Numerical Computations\label{sec:Numerics}}

The approach taken in \cite{key-3} as well as ours to the pair correlation
conjecture for integers representable as the sum of two squares, stated
in (\ref{eq:Keating}), is essentially heuristic, and so some numerical
computations are in place in order to support our conjecture. The
conjecture as stated here is that as $n\rightarrow\infty$ 
\[
\frac{y(n)^{2}}{Y_{h}(n)}\rightarrow1
\]
which, as shown in $\eqref{eq:pair normalization}$, can be calculated
by taking the ratio between the numeric density of pairs and the conjectured
pair correlation function:\\
 
\[
\frac{y(n)^{2}}{Y_{h}(n)}={\displaystyle \frac{\frac{1}{n}\#\left\{ m\leq n|\, m\mbox{ and \ensuremath{m+h}}\mbox{ are representable}\right\} }{\frac{1}{\log n}\cdot2W_{2}(h)\prod\limits _{\begin{array}{c}
\substack{\eqm p34\\
p|h
}
\end{array}}\dfrac{1-p^{-(m_{p}(h)+1)}}{1-p^{-1}}}}
\]

In Figure 7.1 we present some calculations of this ratio for various
$h$ :

%
\begin{figure}[H]
\includegraphics[scale=0.5]{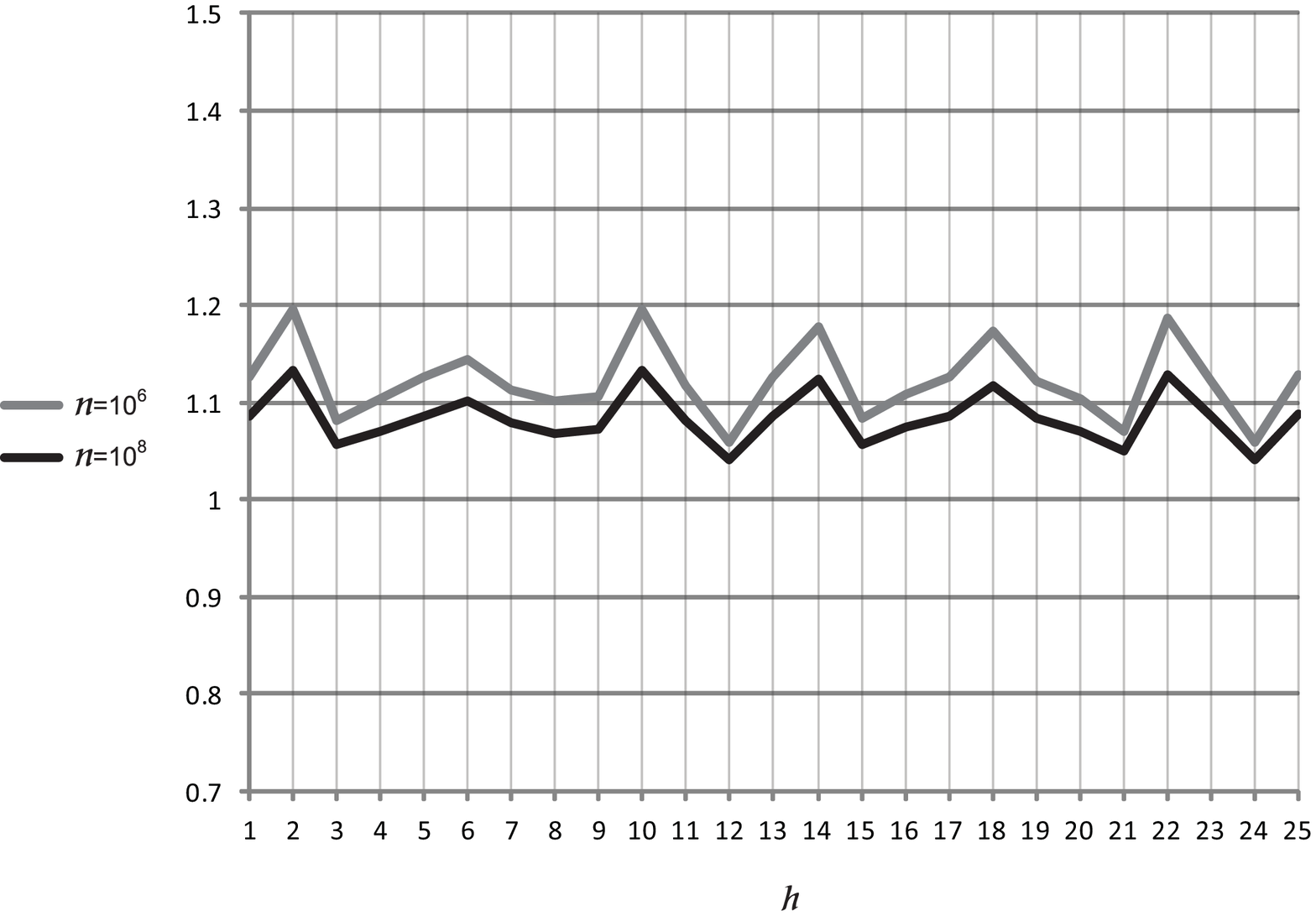}
\caption{$\frac{y(n)^{2}}{Y_{h}(n)}$ for $1\leq h\leq25$ at $n=10^{6},$$10^{8}$}
\end{figure}

Examining different values of $h$ for which the primes $2$ and $\eqm p34$
appear with equal multiplicity, such as $h=1,5,17,25$ or $h=4,20$,
one can see they take very similar values. This was checked for many
more values of $h$ which are not shown here and so strengthens our
belief that the pair correlation depends only on the multiplicity
of these primes in $h$. 

One can also see that the fluctuations between different values of
$h$ diminish for larger $n$, where the peaks in the above graph
are obtained at values of $h$ for which $m_{2}(h)=1,2$ or $m_{3}(h)=1$,
since the small primes are the most dominant in our computations. 

We must not be discouraged by the extremely slow decay to $1$, for
it is consistent with the large error term which appears in Landau's
theorem in (\ref{eq:Landau's theorem}). In fact the convergence implied
in Landau's theorem, or more precisely 
\[
\b(n)^{2}={\displaystyle \left(\frac{\#\left\{ m\leq n|\, m\mbox{ is representable}\right\} }{\b\frac{n}{\sqrt{\log n}}}\right)^{2}}\rightarrow1
\]
as $n\rightarrow\infty$, shows similar behavior as shown in Figure
$7.2$, in which the values for the ratio $\frac{y(n)^{2}}{Y_{h}(n)}$
are calculated for $h=1$. The reason we compare the rate of convergence
to that of $\b(n)^{2}$ and not to $\b(n)$ is that we look at pairs
of representable integers. The values for the ratio $\frac{y(n)^{2}}{Y_{h}(n)}$
are calculated for $h=1$.

%
\begin{figure}[H]
\includegraphics[scale=0.5]{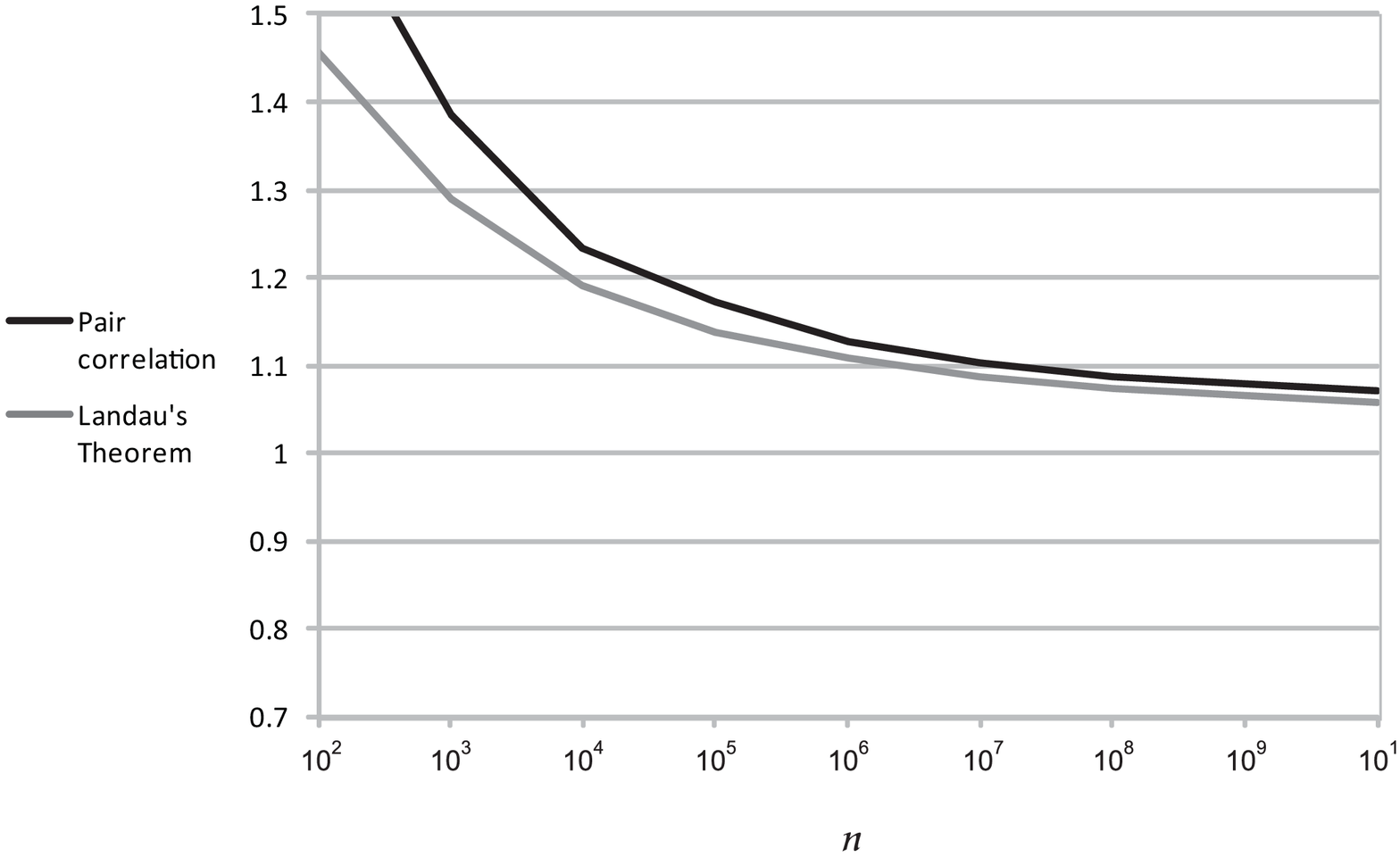}
\caption{$\frac{y(n)^{2}}{Y_{1}(n)}$ and Landau's $\b^{2}$ convergence}
\end{figure}

The generalizations presented in Section \ref{sec:Generalization-to-Other}
for integers of the form $x^{2}+dy^{2}$ show similar numeric results.
Figure $7.3$ is the equivalent of Figure $7.1$ for integers representable
by $x^{2}+2y^{2}$.
\begin{figure}[H]
\includegraphics[scale=0.5]{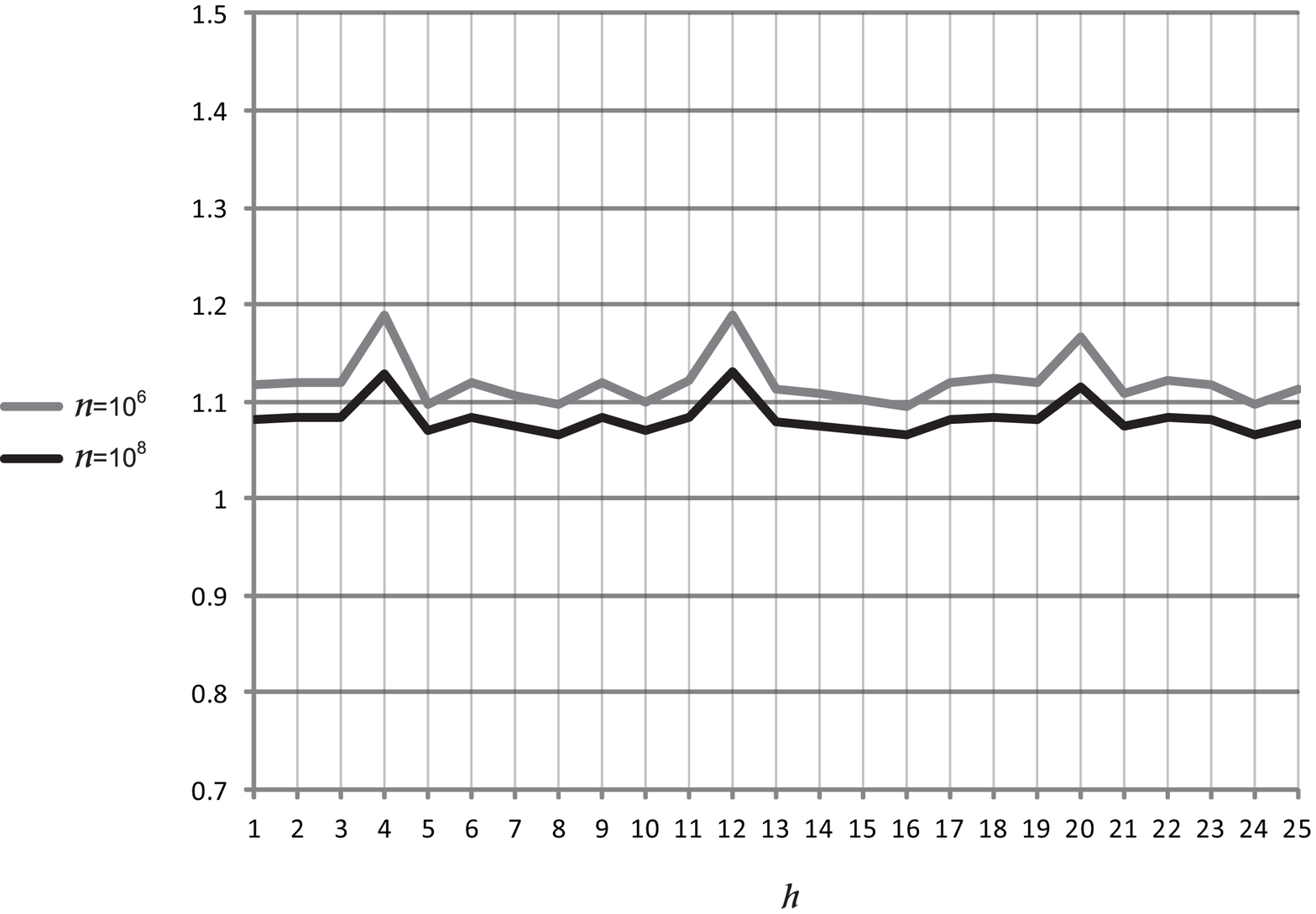}
\caption{$\frac{y_{2}(n)^{2}}{Y_{2,h}(n)}\mbox{ for}1\leq h\leq25\mbox{ at }n=10^{6},10^{8}$}
\end{figure}

We have obtained results of this type for the other forms in question
where the main difference between the forms is the location of the
peaks, which occur at values of $h$ with small $m_{p}(h)$ for small
primes $p\in Q_{d}\cup R_{d}$.

To conclude we have arrived with numerical results which are consistent
with our expectation regarding the dependency on the prime decomposition
of $h$, and regarding the rate of convergence. Note that the numerical
data presented here improves previous computations by a factor of
$10$ for $1\leq h\leq25$ as appears in Figure $7.1$ and by $1000$
for $h=1$ as appears in Figure $7.3$.

\section{further Directions}

The work presented here may be expanded by producing conjectures for
$k-$correlation functions for the set of representable pairs for
$k\geq3$, as described in (\ref{eq:Hardy-LittlewoodSumsofsquares}).
For example, following the methods presented for the calculation of
the mean density and the pair correlation one can derive the following
result for the density of representable triplets of the form $\left(m,m+1,m+2\right)$
for $m\leq n$, given by 
\begin{equation}
\frac{1}{\log^{\frac{3}{2}}n}\cdot\frac{1}{8\b}\cdot\prod_{\begin{subarray}{c}
\eqm p34\\
m\leq n
\end{subarray}}\frac{1-\frac{2}{p}}{\left(1-\frac{1}{p}\right)^{2}}\thickapprox\frac{0.11698}{\log^{\frac{3}{2}}n}\label{eq:triplets}
\end{equation}

It seems possible to generalize this result for triplets $\left(m,m+h_{1},m+h_{2}\right)$
and so on for higher degrees, though it would be difficult to obtain
a general $k-$correlation function this way because of the inductive
element of our approach. Also when comparing the expression for the
density of representable triplets (\ref{eq:triplets}) to the expression
for the density of representable pairs (\ref{eq:Keating}) one can
easily notice that the product for the latter depends only on primes
dividing $h$ , where in the case of the triplets the product is over
all primes $\eqm p34$ and so the manipulation of such expressions
is bound to be more complicated. 

A second direction, assuming a $k-$correlation function is obtained,
is to prove (\ref{eq:Gallagher's lemma sum of squares}), which is
a version of Gallagher's Lemma (\ref{eq:Gallagher's lemma primes})
for sums of two squares. Gallagher's approach in \cite{key-4}, and
similarly the approach taken by Ford in \cite{key-10} when proving
the Lemma in the case of the primes, would apparently not do in the
case of sums of two squares. I addition it is important to note that
our proof of Gallagher's Lemma for sums of two squares and $k=2$
uses the methods of the analytic theory of Dirichlet series, and these
methods become extremely difficult in higher dimensions. This means
that even for $k=3,4$ a new approach for Gallagher's Lemma for sums
of squares must be found. For these reasons we did not pursue any
additional $k-$correlation conjectures.

It is important to note that the main difference between the case
of the set of primes and the case of the set of integers representable
as a sum of two squares is the $k-$correlation conjectures. Hardy
and Littlewood's conjecture for primes (\ref{eq:Hardy-Littlwood})
depends only on $\nu_{\mathbf{d}}(p)$, which stands for the number
of residue classes modulo $p$ occupied by $d_{1},...,d_{k}$, which
in the case of $k=2$ is equivalent to whether or not $p$ divides
$d_{2}-d_{1}$ or in other words whether or not $m_{p}(d_{2}-d_{1})$
is $0$. In the case of the sums of squares Connors and Keating's
conjecture (\ref{eq:Keating}) and the numerical work presented in
Section (\ref{sec:Numerics}) provide evidence of dependence also
on the values of $m_{p}(d_{2}-d_{1})$.

\end{document}